\documentclass[12pt]{amsart}

\usepackage{amsfonts,latexsym,amsmath,amsthm,amssymb, color}
\usepackage[margin=1.0in]{geometry}

\def\N{{\mathbb N}}

\def\R{{\mathbb R}}
\def\Z{{\mathbb Z}}
\def\ddd{{\mathcal{D}}}

\def\f{{\mathcal F}}

\newcommand{\scal}[1]{\langle#1\rangle}

\newtheorem*{theorem*}{Theorem}
\numberwithin{equation}{section}

\newtheorem{theorem}{Theorem}[section]
\newtheorem{proposition}[theorem]{Proposition}
\newtheorem{lemma}[theorem]{Lemma}
\newtheorem{corollary}[theorem]{Corollary}

\DeclareMathOperator{\supp}{supp}
\DeclareMathOperator{\far}{far}
\DeclareMathOperator{\inte}{int}
\DeclareMathOperator{\cc}{cc}
\DeclareMathOperator{\bc}{bc}
\DeclareMathOperator{\PW}{PW}
\DeclareMathOperator{\BMO}{BMO}
\DeclareMathOperator{\ri}{ri}
\DeclareMathOperator{\dist}{dist}
\DeclareMathOperator{\pv}{p.v.}

\begin{document}

\title[Nehari's theorem in several variables]{Nehari's theorem for convex domain Hankel and Toeplitz operators in several variables}
\author{Marcus Carlsson}
\address{Centre for Mathematical Sciences, Lund University\\Box 118, SE-22100, Lund,  Sweden\\}
\email{mc@maths.lth.se}

\author{Karl-Mikael Perfekt}
\address{Department of Mathematics and Statistics,
	University of Reading, Reading RG6 6AX, United Kingdom}
\email{k.perfekt@reading.ac.uk}

\begin{abstract} We prove Nehari's theorem for integral Hankel and Toeplitz operators on simple convex polytopes in several variables. A special case of the theorem, generalizing the boundedness criterion of the Hankel and Toeplitz operators on the Paley--Wiener space, reads as follows. Let $\Xi = (0,1)^d$ be a $d$-dimensional cube, and for a distribution $f$ on $2\Xi$, consider the Hankel operator
	$$\Gamma_f (g)(x)=\int_{\Xi} f(x+y) g(y) \, dy, \quad x \in\Xi.$$
Then $\Gamma_f$ extends to a bounded operator on $L^2(\Xi)$ if and only if there is a bounded function $b$ on $\R^d$ whose Fourier transform coincides with $f$ on $2\Xi$. This special case has an immediate application in matrix extension theory: every finite multi-level block Toeplitz matrix can be boundedly extended to an infinite multi-level block Toeplitz matrix. In particular, block Toeplitz operators with blocks which are themselves Toeplitz, can be extended to bounded infinite block Toeplitz operators with Toeplitz blocks.
\end{abstract}
\maketitle


\section{Introduction}
For an open connected set $\Xi \subset \R^d$, $d \geq 1$, let
$$\Omega = \Xi + \Xi = \{x + y \, : \, x\in \Xi, \, y\in \Xi\},$$
and consider a distribution $f$ defined on $\Omega$. The associated general domain Hankel operator $\Gamma_f = \Gamma_{f, \Xi}$ is the (densely defined) operator $\Gamma_f \colon L^2(\Xi) \to L^2(\Xi)$, given by
\begin{equation*}
\Gamma_f (g)(x)=\int_{\Xi} f(x+y) g(y) \, dy, \quad x \in\Xi,
\end{equation*}
where $dy$ is the Lebesgue measure on $\R^d$.

The case $\Xi = \R_+ = (0, \infty)$ for $d=1$ corresponds to the class of usual Hankel operators; when represented in the appropriate basis of $L^2(\R_+)$, the operator $\Gamma_{f, \R_+}$ is realized as an infinite Hankel matrix $\{a_{n+m}\}_{n,m =0}^\infty$ \cite[Ch. 1.8]{peller2012hankel}. \textit{Nehari's theorem} characterizes the bounded Hankel operators $\Gamma_f \colon L^2(\R_+) \to L^2(\R_+)$. For a function $g$ on $\R^d$, we let
$\hat{g} = \mathcal{F}g$ denote its Fourier transform,
$$\hat{g}(\xi) = \mathcal{F}g(\xi) = \int_{\mathbb{R}^d} g(x) e^{-2\pi i x \cdot \xi} \, dx, \quad \xi \in \mathbb{R}^d.$$
\begin{theorem*}[Nehari \cite{Nehari}]
Suppose that $f$ is a distribution in $\R_+$, $f\in \ddd'(\R_+)$. Then $\Gamma_f \colon L^2(\mathbb{R}_+) \to L^2(\mathbb{R}_+)$ is bounded if and only if there exists a function $b\in L^\infty(\R)$ such that $\hat{b}|_{\R_+}=f$. Moreover, it is possible to choose $b$ so that
\begin{equation} \label{eq:neharicon}
\|\Gamma_f\| = \|b\|_{L^\infty}.
\end{equation}
\end{theorem*}
Nehari's theorem is canonical in operator theory. The two most common proofs proceed either by factorization in the single variable Hardy space or by making use of the commutant lifting theorem.

For $d > 1$, the operators $\Gamma_{f, \R_+^d}$, $\Xi = \R_+^d$, correspond to (small) Hankel operators on the product domain multi-variable Hardy space $H^2_d$. In this case, the analogue of Nehari's theorem remains true, apart from \eqref{eq:neharicon}, but it is significantly more difficult to prove. It was established by Ferguson and Lacey ($d=2$) and Lacey and Terwilleger ($d > 2$) \cite{ferguson2002characterization,MR2491875}. A precise statement is given in Theorem~\ref{lacey}.

The main purpose of this article is to prove Nehari's theorem when $\Xi \subset \R^d$ is a simple convex polytope.
{
	\renewcommand{\thetheorem}{\ref{t1alt}}
	\begin{theorem}
			Let $\Xi$ be a simple convex polytope, and let $f\in \ddd'(\Omega)$ where $\Omega=2\Xi$. Then $\Gamma_f \colon L^2(\Xi) \to L^2(\Xi)$ is bounded if and only if there is a function $b \in L^\infty(\R^d)$ such that $\hat b |_{\Omega}=f.$ There exists a constant $c > 0$, depending on $\Xi$, such that $b$ can be chosen to satisfy
			$$c\|b\|_{L^\infty} \leq \|\Gamma_f\| \leq \|b\|_{L^\infty}.$$
	\end{theorem}
	\addtocounter{theorem}{-1}
}
When $d = 1$, the only open connected sets $\Xi \subset \R$ are the intervals $\Xi = I$. In this case, Theorem~\ref{t1alt} is due to Rochberg \cite{rochberg1987toeplitz}, who called the corresponding operators $\Gamma_{f, I}$ Hankel/Toeplitz operators on the Paley--Wiener space. They have also been called Wiener--Hopf operators on a finite interval \cite{Pel88}. These operators have inspired a wealth of theory in the single variable setting -- see Section~\ref{subsec:PW}, where we shall interpret Theorem~\ref{t1alt} in the context of Paley--Wiener spaces.

Even for $d=1$, our proof of Theorem~\ref{t1alt} appears to be new. However, in several variables our proof relies on the Nehari theorem of Ferguson--Lacey--Terwilleger, and can therefore not be used to give a new proof of their results.

We shall also consider general domain Toeplitz operators $\Theta_f=\Theta_{f, \Xi} \colon L^2(\Xi) \to L^2(\Xi)$. In this context, $f$ is a distribution defined on $\Omega = \Xi - \Xi$, and $\Theta_f$ is densely defined via
$$\Theta_f(g)(x) = \int_{\Xi} f(x-y) g(y) \, dy, \quad  x \in\Xi.$$

If $\Xi$ after a translation is invariant under the reflection $x \mapsto -x$, then the classes of Hankel operators $\Gamma_{f, \Xi}$ and Toeplitz operators $\Theta_{\widetilde{f}, \Xi}$ are essentially the same, and Theorem~\ref{t1alt} immediately yields a boundedness result. This reasoning is applicable to the cube $\Xi = (0,1)^d$, for example.
{
	\renewcommand{\thetheorem}{\ref{t2}}
\begin{corollary}
	Let $\Xi$ be a simple convex polytope such that for some $z \in \R^d$ it holds that $\Xi+z = -\Xi - z$. Let $f \in \ddd'(\Omega)$, $\Omega=\Xi - \Xi = 2\Xi + 2z$. Then $\Theta_f$ is bounded if and only if there exists a function $b \in L^\infty(\R^d)$ such that $\hat b|_{\Omega} = f$. There exists a constant $c > 0$, depending on $\Xi$, such that $b$ can be chosen to satisfy
$$c\|b\|_{L^\infty} \leq \|\Theta_f\| \leq \|b\|_{L^\infty}.$$
\end{corollary}
	\addtocounter{theorem}{-1}
}
On the other hand, when $\Xi$ is a proper convex unbounded set, containing an open cone say, it is clear that the boundedness characterizations of $\Theta_{f, \Xi}$ and $\Gamma_{f,\Xi}$ may be completely different; plainly explained by the fact that $\Omega = \Xi - \Xi = \R^d$ in the Toeplitz case, while $\Omega = \Xi + \Xi = 2\Xi \subsetneq \R^d$ for Hankel operators. In this setting, identifying the boundedness of $\Theta_f$ carries none of the subtleties of Nehari-type theorems. In Theorem~\ref{toeplitz} we obtain the expected boundedness result for a class of ``cone-like'' domains $\Xi$. Rather than giving a precise statement here, let us record the following corollary of Theorem~\ref{toeplitz}.
{
	\renewcommand{\thetheorem}{\ref{cor:ex}}
\begin{corollary}
Let $\Xi \subset \R^d$ be any open connected domain such that
$$(1,\infty)^d \subset \Xi \subset (0,\infty)^d,$$
and let $f \in \ddd'(\R^d)$. Then $\Theta_{f} \colon L^2(\Xi) \to L^2(\Xi)$ is bounded if and only if $f$ is a tempered distribution and $\|\hat{f}\|_{L^\infty(\R^d)} < \infty$, and in this case
$$\|\Theta_f\| = \| \hat{f} \|_{L^\infty}.$$
\end{corollary}
	\addtocounter{theorem}{-1}
}
In the final part of the paper we shall give an application of Theorem~\ref{t1alt} to matrix completion theory, essentially obtained by discretizing Corollary~\ref{t2} when $\Xi$ is a cube. To avoid introducing further notation, we shall only state the result in words for now. Recall that a Toeplitz matrix is one whose diagonals are constant. An $N \times N$ $d$-multilevel block Toeplitz matrix is an $N \times N$ Toeplitz matrix whose entries are $N\times N$ $(d-1)$-multilevel block Toeplitz matrices. Here $N$ could be finite or infinite. A $1$-multilevel block Toeplitz matrix is simply an ordinary Toeplitz matrix. A $2$-multilevel block Toeplitz matrix is what is usually considered a block Toeplitz matrix where each block itself is Toeplitz.
{
	\renewcommand{\thetheorem}{\ref{thm:blocktoep}}
	\begin{theorem}
	Every finite $N \times N$ $d$-multilevel block Toeplitz matrix can be extended to an infinite $d$-multilevel block Toeplitz matrix bounded on $\ell^2$, with a constant which only depends on the dimension $d$.
	\end{theorem}
	\addtocounter{theorem}{-1}
}
For scalar Toeplitz matrices ($d=1$) this result is well-known \cite{BT01, NF03, Sara07, Vol04}, although not as firmly cemented in the literature as the Nehari theorem itself; see \cite[Ch.~V.2, V.8]{nik} for a proof based on Parrot's lemma and a discussion of the result's history. For $d=1$, the converse deduction of Theorem~\ref{t1alt} starting from Theorem~\ref{thm:blocktoep} can be found in \cite{carlsson2011truncated}.

The paper is laid out as follows. In Section~\ref{sec:background} we will give a more formal background and introduce necessary notation. We will also discuss the relationship between $\Gamma_{f, \Xi}$, Paley--Wiener spaces, and co-invariant subspaces of the Hardy spaces. In Section~\ref{secdist} we will prove approximation results for distribution symbols with respect to Hankel and Toeplitz operators, allowing us to reduce to smooth symbols. Section~\ref{sec:convex} briefly outlines what we need to know about convex sets and polytopes. In Section~\ref{sec:gdhankel} we prove Theorem~\ref{t1alt}, our Nehari theorem for Hankel operators. We also indicate how the proof extends to certain unbounded polyhedral domains. In Section~\ref{sec:gdtoep} our main result on Toeplitz operators is shown, Theorem~\ref{toeplitz}. Finally, Section~\ref{sec:blocktoep} gives the proof of Theorem~\ref{thm:blocktoep}.

\section{Further background and related results} \label{sec:background}

\subsection{Hankel operators on multi-variable Hardy spaces} \label{sec:multihankelx}
Let us begin by placing Hankel operators $\Gamma_f$ into the context of classical Hankel operators on Hardy spaces. As before, for $g \in L^2(\mathbb{R}^d)$, let $\hat{g} = \mathcal{F}g$ denote its Fourier transform,
$$\hat{g}(\xi) = \mathcal{F}g(\xi) = \int_{\mathbb{R}^d} g(x) e^{-2\pi i x \cdot \xi} \, dx, \quad \xi \in \mathbb{R}^d.$$
 For the inverse transform we write $\mathcal{F}^{-1}(g)=\check g$.
 The product domain Hardy space $H^2_d$ is the proper subspace of $L^2(\mathbb{R}^d)$ of functions whose Fourier transforms are supported in the cone $\overline{\R_+^d}$, $\R_+ = (0,\infty)$,
 $$ H^2_d = \left\{G \in L^2(\mathbb{R}^d) \, : \, \supp \hat{G} \subset \overline{\mathbb{R}_+^d}\right\}.$$
 We let $P_d \colon L^2(\mathbb{R}^d) \to H^2_d$ denote the orthogonal projection, and let $J \colon L^2(\mathbb{R}^d) \to L^2(\mathbb{R}^d)$ be the involution defined by $JG(x) = G(-x)$, $x \in \mathbb{R}$.

   Consider $\Gamma_f = \Gamma_{f, \Xi}$ for $\Xi = \mathbb{R}_+^d$ with $f \in L^2(\mathbb{R}_+^d)$. For a dense set of $g, h \in L^2(\mathbb{R}_+^d)$ we have that
\begin{equation} \label{eq:hankelform}
\langle \Gamma_f g, h \rangle_{L^2(\mathbb{R}_+^d)} = \langle \check f J\check{g}, \check h\rangle_{H^2_d}.
\end{equation}
It follows that the (possibly unbounded) operator $\Gamma_f \colon L^2(\R_+^d) \to L^2(\R_+^d)$ is unitarily equivalent to the small Hankel operator $Z_{\check{f}} \colon H^2_d \to H^2_d$,
$$Z_{\check{f}} G = P_d(\check{f} \cdot JG).$$
 Note that any $b$ such that $\hat{b}|_{\mathbb{R}_+^d} = f$ generates the same Hankel operator as $\check{f}$, $Z_b = Z_{\check{f}}$.

 To justify the above computation easily we assumed that $f \in L^2(\R_+^d)$. An approximation argument is needed to consider general symbols $f$, which may only be distributions in $\R_+^d$. We provide this later in Proposition~\ref{pextend}. We can then read off the boundedness of $\Gamma_f$ from the boundedness of the corresponding Hankel operator on $H^2_d$. When $d=1$ and $\Xi=\Omega=\R_+$, the analogue of Theorem~\ref{t1alt} is exactly the classical Nehari theorem. In higher dimensions the corresponding theorem is due to Ferguson--Lacey--Terwilleger  \cite{ferguson2002characterization,MR2491875}. In our notation, their results read as follows.
\begin{theorem}\label{lacey}
Suppose $\Xi=\Omega=\R_+^d$ and that $f$ is a distribution in $\R^d_+$, $f\in \ddd'(\R^d_+)$. Then $\Gamma_f \colon L^2(\mathbb{R}^d_+) \to L^2(\mathbb{R}^d_+)$ is bounded if and only if there exists a function $b\in L^\infty(\R^d)$ such that $\hat{b}|_{\R_+^d}=f$. Moreover, there exists a constant $c > 0$, depending on $d$, such that $b$ can be chosen to satisfy \begin{equation}\label{c}c\|b\|_{L^\infty}\leq\|\Gamma_f\|\leq\|b\|_{L^\infty}.\end{equation}
\end{theorem}
For $d > 1$ it is not possible to take $c=1$ in \eqref{c}, see for example \cite{ortega2012lower}.

\subsection{Hankel operators on bounded domains}
We now discuss bounded domains $\Xi$, the setting of our main result. The only convex bounded domains in $\R$ are the intervals $I\subset \R$. Translations, dilations, and reflections carry the operator $\Theta_{f,I}$ onto $\Gamma_{\tilde{f},J}$, where $J \subset \mathbb{R}$ is any other interval and $\tilde{f}$ arises from transforming $f$ appropriately. In one variable it thus suffices to consider operators $\Gamma_{f,(0,1)}$ where $\Xi = (0,1)$. Rochberg \cite{rochberg1987toeplitz} called these operators Hankel operators on the Paley-Wiener space and proved Theorem~\ref{t1alt} in the one-dimensional case.

In the same article \cite{rochberg1987toeplitz}, it is posed as an open problem to characterize the bounded Hankel operators $\Gamma_{f,\Xi}$ when $\Xi$ is a disc in $\mathbb{R}^2$. We are not able to settle this question, but Theorem~\ref{t1alt} does provide the answer when $\Xi = (0,1)^d$ is a cube in $\R^d$. As we will see, the Hankel operators $\Gamma_{f, (0,1)^d}$ constitute a natural generalization of the Hankel operators on the Paley-Wiener space. On a technical level, the reason that we are able to prove Theorem~\ref{t1alt} when $\Xi$ is a simple convex polytope, but not when $\Xi$ is a ball, is that we rely on Theorem~\ref{lacey}. In applying Theorem~\ref{lacey} to our situation, the corners of the boundary of $\Xi$ are actually of help rather than hindrance. We consider the case of a ball to be an interesting open problem for which we do not dare to make a firm conjecture. In view of Fefferman's disproof of the disc conjecture \cite{fefferman1971multiplier}, Nehari theorems might turn out to be quite different for balls and polytopes.

\subsection{Toeplitz operators}
When $d=1$ and $\Xi=\R_+$, $\Omega = \R$, the operators $\Theta_f$ are known as Wiener-Hopf operators \cite[Ch. 9]{bottcher2013analysis}. Analogously with Hankel operators, these can be shown to be unitarily equivalent to Toeplitz matrix operators on $\ell^2(\N)$. In this case the boundedness characterization is easy to both state and prove, \begin{equation}\label{i9}\|\Theta_f\|=\|\hat f\|_{L^\infty}.\end{equation}
In Theorem~\ref{toeplitz} we extend \eqref{i9} to Toeplitz operators $\Theta_{f, \Xi}$ for a class of ``cone-like'' domains $\Xi \subset \R^d$, for which $\Omega = \Xi - \Xi = \R^d$.

\subsection{Truncated correlation operators}\label{subsec:TCO}
For open connected sets $\Xi, \Upsilon \subset \R^d$ it is also convenient to introduce the more general ``truncated correlation operators'' $\Psi_{f, \Upsilon, \Xi} \colon L^2(\Upsilon)\rightarrow L^2(\Xi)$, defined by
$$\Psi_{f}(g)({x})=\int_{\Upsilon}f({x}+{y}) g({y}) ~d{y},\quad {x}\in\Xi,$$
 where $f$ lives on $\Omega=\Xi+\Upsilon$. This class of operators includes both general domain Hankel and Toeplitz operators, by letting $\Upsilon = \Xi$ and $\Upsilon = -\Xi$, respectively.

 For our purposes, general truncated correlation operators will only appear in intermediate steps toward proving the main results, but they also carry independent interest. They were introduced in \cite{andersson2015general}, where their finite rank structure was investigated. In \cite{andersson2017fixed} it was shown that they have a fundamental connection with frequency estimation on general domains, motivating the practical need for understanding such operators not only on domains of simple geometrical structure. In \cite{andersson2016structure} it is explained how one may infer certain results for the integral operators $\Psi_f$ from their discretized matrix counterparts. We warn the reader that in naming the operators $\Gamma_f$, $\Theta_f$, and $\Psi_f$ we have slightly departed from previous work, reserving the term (general domain) Hankel operator for truncated correlation operators of the form $\Psi_{f,\Xi,\Xi}$.

\subsection{Hankel operators on multi-variable Paley--Wiener spaces} \label{subsec:PW}
Another viewpoint is offered through co-invariant subspaces of the Hardy spaces $H^2_d$. For a domain $\Xi \subset \R^d$, let $\PW_\Xi$ denote the subspace of $L^2(\R^d)$ of functions with Fourier transforms supported in $\Xi$,
$$\PW_\Xi = \{G \in L^2(\R^d) \, : \, \supp \hat{G} \subset \overline{\Xi} \}.$$
In the classical case $\Xi = (0,1) \subset \R$, note that
$$\PW_{(0,1)} = H^2_1 \ominus \{G \in H^2_1 \, : \, \supp \hat{G} \subset [1,\infty)\} = H^2_1 \ominus \theta H^2_1,$$
where
$$\theta(x) = e^{i2\pi x}, \quad x \in \mathbb{R}.$$
Hence $\PW_{(0,1)}$ is the ortho-complement (in $H^2_1$) of $\theta H^2_1$, the shift-invariant subspace of $H^2_1$ with inner factor $\theta$. This space is usually denoted $K_\theta$,
$$\PW_{(0,1)} = K_\theta := (\theta H^2_1)^\perp.$$
By a calculation similar to \eqref{eq:hankelform} we see that
$\Gamma_{f, (0,1)}$ is unitarily equivalent to the compression of the Hankel operator $Z_{\check f}$ to $\PW_{(0,1)}$,
$$\Gamma_{f, (0,1)} \simeq P_{\PW_{(0,1)}} Z_{\check{f}}|_{\PW_{(0,1)}},$$
where $P_{\PW_{(0,1)}} \colon H^2_1 \to \PW_{(0,1)}$ denotes the orthogonal projection onto $\PW_{(0,1)}$. Such \textit{truncated Toeplitz and Hankel operators} are now very well studied on general $K_\theta$-spaces \cite{BBK11, BCFMT, Bess152, Bess15, CGRW10, FR75, Nik86, Pel88, Sara07}.

In the case of the cube $\Xi = (0,1)^d \subset \R^d$, $d > 1$, the Hankel operator $\Gamma_{f,\Xi}$ may, just as for $d=1$, be understood as the compression of a Hankel operator to a co-invariant subspace of $H^2_d$. Namely,
$$\PW_{(0,1)^d} = \{G \in H_d^2 \, : \, \supp \hat{G} \subset [0,1]^d\} = \{G \in H_d^2 \, : \, \supp \hat{G} \subset \overline{\R_+^d} \setminus (0,1)^d\}^\perp.$$
If $G \in H_d^2 \cap L^\infty(\R^d)$, it is clear that $G\PW_{(0,1)^d}^\perp \subset \PW_{(0,1)^d}^\perp$, since
$$\mathcal{F}(GH)(\xi) = \int_{\R_+^d} \hat{G}(y) \hat{H}(\xi-y) \, dy = 0, \quad H \in \PW_{(0,1)^d}^\perp, \; \xi \in [0,1]^d.$$
Hence $\PW_{(0,1)^d}^\perp \subset H^2_d$ is an invariant subspace (under multiplication by bounded holomorphic functions), and as before we have that
$$\Gamma_{f, (0,1)^d} \simeq P_{\PW_{(0,1)^d}} Z_{\check{f}}|_{\PW_{(0,1)^d}},$$
where $P_{\PW_{(0,1)^d}} \colon H^2_d \to \PW_{(0,1)^d}$ denotes the orthogonal projection onto $\PW_{(0,1)^d}$.

Finally, let us briefly discuss the viewpoint of weak factorization. The Hardy space $H^1_d$ is defined as the closure of $\mathcal{F}^{-1}(C_c^\infty(\R_+^d))$ in $L^1(\R^d)$. Similarly, we define $\PW^1_\Xi$ as the closure of $\mathcal{F}^{-1}(C_c^\infty(\Xi))$ in $L^1(\R^d)$. As is well known, see for example \cite[Theorem 6.4]{lacey2007lectures}, Theorem~\ref{lacey} is equivalent to the fact that $H^1_d$ is the projective tensor product of two copies of $H^2_d$,
\begin{equation} \label{eq:weakfac}
H^1_d = H^2_d \odot H^2_d,
\end{equation}
with equivalence of norms. Here the projective tensor product norm on $X \odot X$, $X$ a Banach space of functions, is given by
$$\|G\|_{X \odot X} = \inf \left \{ \sum_j \|G_j\|_{X} \|H_j\|_{X} \, : \, G = \sum_j G_j H_j, \; G_j, H_j \in X \right\},$$
$X \odot X$ being defined as the completion of finite sums $\sum_j G_j H_j$ in this norm.

The reason that Theorem~\ref{lacey} is equivalent to \eqref{eq:weakfac} is the following: by \eqref{eq:hankelform}, $\Gamma_{f, \R_+^d}$ is bounded if and only if
$$|\langle \check{f}, GH \rangle_{H^2_d} | \leq C\|G\|_{H^2_d} \|H\|_{H^2_d},$$
which means precisely that $\check{f}$ induces a bounded functional on $H^2_d \odot H^2_d$, $\check{f} \in (H^2_d \odot H^2_d)^*$. On the other hand, the existence of $b \in L^\infty(\R^d)$ such that $\hat b|_{\R_+^d} = f|_{\R_+^d}$, so that $\langle \check{f}, GH \rangle_{H^2_d} = \langle b, GH \rangle_{H^2_d}$, $G,H \in H^2_d$, means, by the Hahn--Banach theorem, precisely that $\check{f} \in (H_d^1)^*$.

Theorem~\ref{t1alt} yields a similar weak factorization theorem for Paley--Wiener spaces. We postpone the proof to Section~\ref{sec:gdhankel}, mentioning only that corresponding weak factorization for $K_\theta$-spaces plays an important role in \cite{BBK11} and \cite{Bess152}.
{
	\renewcommand{\thetheorem}{\ref{cor:fact}}
\begin{corollary}
	Let $\Xi$ be a simple convex polytope, and let $\Omega = 2\Xi$. Then
	$$PW_{\Omega}^1 = \PW_{\Xi}\odot \PW_{\Xi}.$$
	The norms of these Banach spaces are equivalent.
\end{corollary}
	\addtocounter{theorem}{-1}
}

\subsection{Brief historical overview}\label{secBHO}
Z. Nehari published his famous theorem in 1957 \cite{Nehari}, inspiring the search for analogous statements in other contexts; positive results are themselves often referred to as Nehari theorems. The most natural inquiries are perhaps those related to Hankel operators on Hardy spaces of several variables.
Nehari's theorem for the Hardy space of the unit ball was proven by Coifman, Rochberg and Weiss in 1976 \cite[Thm. VII]{coifman1976factorization}, but this setting is rather different from the one considered in this paper. 

For the product domain Hardy space $H^2_d$, Hankel operators can be defined by either projecting on $H^2_d$ or on the larger space $L^2(\R^d) \ominus H^2_d$. The first option leads to the ``small'' Hankel operators considered in Section~\ref{sec:multihankelx}, while the second type of operator is commonly referred to as a ``big'' Hankel operator. In the notation of Section~\ref{subsec:TCO}, a small Hankel operator is an operator $\Psi_{f,\R_+^d,\R_+^d}=\Gamma_{f,\R_+^d}$, whereas big Hankel operators are of the form $\Psi_{f,\R_+^d,\R^d\setminus\overline{\R_+^d}}$. 
When transferred to operators on the Hardy space of the polydisc, small Hankel operators correspond, in the standard basis, to infinite matrices with a certain block Hankel structure (cf. Section~\ref{sec:blocktoep}). 

The big Hankel operators were extensively studied by Cotlar and Sadosky. In particular, boundedness of the big Hankel operators was characterized in terms of certain $\BMO$ type estimates in \cite{cotlar1994nehari}. Small Hankel operators were investigated by Janson and Peetre \cite{janson1988paracommutators} in 1988. They introduced ``generalized Hankel and Toeplitz operators'' as particular cases of a more general class of pseudo-differential operators called paracommutators. In their terminology, an operator of the form $\Psi_{f,\Xi,\Upsilon}$ is a generalized Hankel operator if $\Xi$ and $\Upsilon$ are open cones and $\overline{\Xi} \cap(-\overline{\Upsilon})=\{0\}$, whereas it is called Toeplitz if $\Xi \cap (-\Upsilon) \neq \emptyset.$ Hence the general domain Hankel operators $\Gamma_{f,\Xi}$ are generalized Hankel operators a l\'{a} Janson--Peetre whenever $\Xi$ is a cone with mild restrictions, while $\Theta_{f,\Xi}$ is a generalized Toeplitz operator a l\'{a} Janson--Peetre for every open cone $\Xi$. In the Toeplitz case, a full boundedness characterization is given in \cite{janson1988paracommutators}. In the Hankel case, only sufficient conditions for boundedness and Schatten class membership are provided, in terms of $\BMO$ and Besov spaces, respectively.

 As previously mentioned, R. Rochberg considered Hankel operators for bounded domains in 1987 \cite{rochberg1987toeplitz}, studying the case of a finite interval in one dimension. Furthermore, he posed as an open problem to understand the case when $\Xi \subset \R^2$ is a disc. In this latter setting, L. Peng \cite{lishong1987hankel} characterized when $\Gamma_{f, \Xi}$ belongs to the Schatten class $S_p$, for $1 \leq p \leq 2$, in terms of certain Besov spaces adapted to the disc. L. Peng also carried out a similar study \cite{peng1989hankel} for the case of the multidimensional cube, $\Xi = (-1,1)^d$, describing membership in $S_p$ for all $p$, $0 < p < \infty$, as well as giving a sufficient condition for boundedness.

Since then it seems that the field did not see progress until the results of Ferguson--Lacey--Terwilleger \cite{ferguson2002characterization,MR2491875} settled the issue of boundedness of small Hankel operators.

\section{Distribution symbols}\label{secdist}

Let $\Xi,\Upsilon \subset \R^d$ be any open connected sets and let $f\in \ddd'(\Omega)$ be a distribution on $\Omega$, $\Omega=\Xi+\Upsilon$.  We follow the notation of \cite{hormander} in our use of distributions.  We then define the truncated correlation operator $\Psi_f$ as an operator  $\Psi_{f, \Upsilon, \Xi}:C^\infty_c(\Upsilon)\rightarrow C^\infty(\Xi)$ by the formula
$$\Psi_f(\varphi)(x)=({f,T_x\varphi}), \quad x \in \Xi,$$
where $(f,\varphi)$ denotes the action of $f$ on $\varphi$\footnote{We reserve the notation $\scal{f,\varphi}$ for scalar products which are anti-linear in the second entry.} and
$$T_x\varphi(\cdot)=\varphi(\cdot-x).$$
Since $T_x\varphi$ is compactly supported in $\Omega$ for $x \in \Xi$, it follows that $\Psi_f(\varphi)$ this is well-defined and smooth in $\Xi$ (see e.g. \cite[Theorem 4.1.1]{hormander}). Since $C^\infty_c(\Upsilon)$ is dense in $L^2(\Upsilon)$, $\Psi_f$ gives rise to a densely defined operator on the latter space which extends to a bounded operator $\Psi_f \colon L^2(\Upsilon) \to L^2(\Xi)$ if and only if $$\|\Psi_f\|=\sup\left\{\frac{\|\Psi_f(\varphi)\|_{L^2(\Xi)}}{\|\varphi\|_{L^2(\Upsilon)}}:~\varphi\in C_c^\infty(\Upsilon), ~\varphi\neq 0\right\}<\infty.$$
It is clear that $\Psi_f(\varphi)(x) = \int f(x+y)\varphi(y) \, dy$ whenever $f \in L^1_{\textrm{loc}}(\Omega)$. By slight abuse of notation, we write the action of $\Psi_f$ in this way even when $f$ is not locally integrable.

The central question in this paper is the following: \textit{for which domains $\Upsilon$ and $\Xi$ is the boundedness of $\Psi_f$ equivalent to the existence of a function $b\in L^\infty(\R^d)$ such that $\hat b|_{\Omega}=f$?}  Some care must be taken in interpreting this question. For example, the prototypical example of a bounded Hankel operator is the Carleman operator
$$\Gamma_{1/x, \R_+} = \Psi_{1/x, \R_+, \R_+}.$$
The symbol $f(x) = \frac{1}{x}\chi_{\R_+}(x)$ is in this case not a tempered distribution on $\mathbb{R}$ (so the meaning of $\check f$ is unclear) -- it is, however, the restriction of the tempered distribution $\pv \frac{1}{x}$ to $\R_+$. An example with a delta function makes it clear that it is not necessary for $f$ to be locally integrable in $\Omega$ either.

We first record the answer to our question in the trivial direction.

\begin{proposition}\label{ptrivial}
Consider any connected open domains $\Xi,$ $\Upsilon \subset \R^d$, with associated domain $\Omega=\Upsilon+\Xi$. Let $b\in L^\infty(\R^d)$ be given and suppose $f=\hat{b}|_{\Omega}$. Then $\Psi_f \colon L^2(\Upsilon) \to L^2(\Xi)$ is bounded and \begin{equation}\label{inequality}\|\Psi_f\|\leq\|b\|_{L^\infty}.\end{equation}
\end{proposition}
\begin{proof} For $\varphi\in C^\infty_c(\Upsilon)$ we have that
\begin{equation*} 
\Psi_{f}(\varphi)=\f M_b J\f^{-1} \varphi|_{\Xi},
\end{equation*}
where $M_b$ is the operator of multiplication by $b$. The statement is obvious from here.
\end{proof}
Next we establish two technical results on the approximation of distribution symbols by smooth compactly supported functions, Propositions~\ref{pextend} and \ref{pextend2}. They will help us to overcome the technical issues mentioned earlier, in particular allowing us to deduce Theorem~\ref{lacey} from the corresponding statements in \cite{ferguson2002characterization, MR2491875}.

Given open connected domains $\Xi,$ $\Upsilon \subset \R^d$, let $\left(\Upsilon_n\right)_{n=1}^\infty$ be an increasing sequence of connected open subdomains $\Upsilon_n \subset \Upsilon$ such that $$\dist(\Upsilon_n,\partial\Upsilon)>1/n, \quad \cup_{n=1}^\infty \Upsilon_n=\Upsilon.$$
Note that $\Omega_n = \Upsilon_n + \Xi$ is also increasing and satisfies
$$\dist(\Omega_n,\partial\Omega)>1/n, \quad \cup_{n=1}^\infty \Omega_n = \Omega.$$

 Let $\psi\in C^\infty_c(\R^d)$ be a fixed non-negative function with compact support in the ball $B(0,1/2)$ such that $\int_{\R^d} \psi(x) \, dx =1$. For $n\geq 1$ let
 $$\psi_n(x)={n^d}\psi(nx),$$
  so that $(\psi_n)_{n=1}^\infty$ is an approximation of the identity. Since $f \in \ddd'(\Omega)$ and $\supp \psi_n \subset B(0,1/2n)$, the convolution $f*\psi_n$ is well-defined as a function in $C^\infty(\Omega_{2n})$. Let $\rho_n$ be a smooth cut-off function which is $1$ in a neighborhood of $\overline{\Omega_n}$ but zero in a neighborhood of $\Omega_{2n}^c$, and note that $\rho_n(f*\psi_n)$ then naturally defines a function in $C^\infty(\R^n)$. Finally, for a non-negative function $\eta \in C^\infty_c(\R^d)$ with $\|\eta\|_{L^2} = 1$, let $\omega = \eta * \tilde{\eta}$, where $\tilde{\eta}(x) = \eta(-x)$. Then $\omega\in C^\infty_c(\R^d)$ and $$\omega(0)= \|\hat \omega\|_{L^1} = 1.$$ Let $\omega_n(x)=\omega(x/n)$. We introduce $$f_n=\omega_n\rho_n (f*\psi_n)$$ as an approximant of $f$, where the role of $\omega_n$ is to enforce compact support in case $\Omega$ is unbounded. By construction, $f_n \in C_c^\infty(\Omega)$ and it is straightforward to check that $f_n \to f$ in $\ddd'(\Omega)$. As for $\Psi_{f_n, \Upsilon_n, \Xi}$, we have the following result.

\begin{proposition}\label{pextend}
Let $\Xi,$ $\Upsilon$ be connected open domains, $\Omega=\Upsilon+\Xi$, and suppose $f\in \ddd'(\Omega)$. For $n \geq 1$, let $\Omega_n=\Upsilon_n+\Xi$ and $f_n$ be constructed as above. Then  $$\|\Psi_{f_n,\Upsilon_n,\Xi}\| \leq \|\Psi_{f,\Upsilon,\Xi}\|.$$
\end{proposition}

\begin{proof}
First note that
$$\omega_n(x)=\int_{\R^d} n^d \hat{\omega}(n \xi)e^{2\pi i x \cdot \xi} \, d\xi,$$ the integrand on the right having $L^1$-norm equal to $\|\hat\omega\|_{L^1(\R^d)}$. Letting $g_n=\rho_n (f*\psi_n)$, we have for $\varphi\in C^\infty_c(\Upsilon_n)$ and $x \in \Xi$ that
$$\Psi_{f_n}(\varphi)(x)=\int_{\Upsilon_n} \int_{\R^d} n^d\hat{\omega}(n\xi)e^{2\pi i (x+y)\cdot\xi}\,d\xi \, g_n(x+y)\varphi(y) \, dy= \int_{\mathbb{R}^d} n^d \hat{\omega}(n \xi){e^{2\pi i \xi \cdot x}}\Psi_{g_n}(\varphi_\xi)(x) \, d\xi,$$
where $\varphi_\xi(y) = e^{2\pi i y\cdot\xi}\varphi(y)$. Since $\|\varphi_\xi\|_{L^2} = \|\varphi\|_{L^2}$ it follows by the triangle inequality (for $L^2$-valued Bochner integrals) that $$\|\Psi_{f_n, \Upsilon_n, \Xi}\|\leq \|\hat\omega\|_{L^1}\|\Psi_{g_n, \Upsilon_n, \Xi}\| = \|\Psi_{g_n, \Upsilon_n, \Xi}\|.$$

This reduces our task to proving that the operators
$$\Psi_{g_n, \Upsilon_n, \Xi} = \Psi_{\rho_n (f*\psi_n), \Upsilon_n, \Xi} = \Psi_{f*\psi_n, \Upsilon_n, \Xi}$$ are uniformly bounded in $n$.  We have for $\varphi\in C^\infty_c(\Upsilon_n)$ and $x \in \Xi$ that
\begin{align*}
\Psi_{f*\psi_n}(\varphi)(x) &=\int_{\R^d} \int_{\R^d} f((x+y)-z)  \psi_n(z) \, dz~\varphi(y) \, dy \\
&= \int_{\R^d}  f(x+z) \int_{\R^d} \psi_n(y-z)\varphi(y)\, dy \,dz=\Psi_f(\widetilde{\psi}_n * \varphi )(x),\end{align*}
where $\widetilde{\psi}_n(x) = \psi_n(-x)$.
Since
$$\|\widetilde{\psi}_n * \varphi\|_{L^2(\Upsilon)} \leq \|\psi_n\|_{L^1} \|\varphi\|_{L^2(\Upsilon_n)} = \|\psi\|_{L^1} \|\varphi\|_{L^2(\Upsilon_n)} = \|\varphi\|_{L^2(\Upsilon_n)},$$
this completes the proof.
\end{proof}
\begin{proof}[Proof of Theorem~\ref{lacey}]
Suppose that $\Gamma_{f, \R_+^d} = \Psi_{f, \Xi, \Upsilon}$ is bounded, where $\Xi = \Upsilon = \R_+^d$. In this case, we let $\Upsilon_n = (2/n, \infty)^d$. By Proposition~\ref{pextend} we then have that
$$\|\Gamma_{f_n, \Upsilon_n}\|\leq \|\Psi_{f_n, \Upsilon_n,\Xi}\| \leq \| \Gamma_{f, \R_+^d}\|, \quad n \geq 1.$$
Since $\Upsilon_n = z_n + \R_+^d$, $z_n = (2/n, \ldots, 2/n)$, we have that
$$\Gamma_{f_n, \Upsilon_n}(g)(x) = \Gamma_{\tilde{f}_n, \R_+^d} (\tilde{g})(x-z_n),$$
 where $\tilde{f}_n(x) = f_n(x+2z_n)$ and $\tilde{g}_n(x) = g(x+z_n)$.
Since $\tilde{f}_n \in L^2(\R_+^d)$, the computation that lead to \eqref{eq:hankelform} is justified, and we conclude from \cite{ferguson2002characterization, MR2491875} that there is $b_n \in L^\infty(\R^d)$ such that
$$\hat b_n|_{2\Upsilon_n} = f_n|_{2\Upsilon_n}, \quad \|b_n\|_{L^\infty} \leq C \| \Gamma_{f, \R_+^d}\|.$$
By Alaoglu's theorem it follows that there is a weak-star convergent subsequence $(b_{n_k})_{k=1}^\infty$ with limit $b \in L^\infty$ having norm less than $C \| \Gamma_{f, \R_+^d}\|$. It remains to prove that $f=\hat{b}|_{\R_+^d}$, i.e. that $(f,\varphi)=(b,\hat{\varphi})$ holds for all $\varphi\in C^\infty_c(\R_+^d)$. However, this is clear from the construction; since $\hat{\varphi}\in L^1$ we have that
\begin{equation*}
(b,\hat{\varphi})=\lim_{k\rightarrow\infty}(b_{n_k},\hat{\varphi})=\lim_{k\rightarrow\infty}(f_{n_k},{\varphi}) =(f,{\varphi}). \qedhere
\end{equation*}
\end{proof}
In Section~\ref{sec:gdtoep} we will consider Toeplitz operators $\Theta_{f, \Xi}$ for which $\Omega = \Xi - \Xi = \R^d$. In this case $f * \psi_n$ is a smooth function defined in all of $\R^d$, and there is no need to multiply with $\rho_n$ or to introduce the subdomains $\Upsilon_n$. In this case we simply let
$$f_n = \omega_n (f * \psi_n).$$
Clearly, $f_n \to f$ in $\ddd'(\R^d)$ and we have, with the exact same proof as for Proposition~\ref{pextend}, the following approximation result.
\begin{proposition}\label{pextend2}
	Let $\Xi,$ $\Upsilon$ be connected open domains for which $\Omega=\Upsilon+\Xi = \R^d$, and suppose $f\in \ddd'(\R^d)$. For $n \geq 1$, let $f_n$ be constructed as above. Then $$\|\Psi_{f_n,\Upsilon,\Xi}\| \leq \|\Psi_{f,\Upsilon,\Xi}\|.$$
\end{proposition}
\section{On convex sets and polytopes} \label{sec:convex}
We recall some basic properties of convex sets. Given an unbounded convex set $\Omega \subset \R^d$ which is either open or closed, its characteristic cone, also known as its recession cone, is the closed set
$$\cc_{\Omega}=\{x\in\R^d \, : \, \Omega+x\R_+\subset\Omega\}.$$
The support function $h_{\Omega}:\R^d\rightarrow (-\infty,\infty]$ is defined by $$h_{\Omega}({\theta})=\sup_{{x}\in \Omega} {x}\cdot{\theta}.$$ We refer to \cite[Sec.~7.4]{hormander} for the basic properties of $h_\Omega$.  The barrier cone of $\Omega$ is the set
\begin{equation}\label{eq22} \bc_{\Omega}=\{\theta \in \R^d \, : \, h_{\Omega}({\theta})<\infty\}.\end{equation}
  The characteristic cone $\cc_{\Omega}$ coincides with the polar cone of the barrier cone $\bc_{\Omega}$, that is,
$$\cc_{\Omega} = \{x \in \R^d \, : \,  x \cdot y \leq 0, \; \: \forall y\in \bc_{\Omega}\}.$$
To give a complete reference for this claim, first note that for closed convex sets $\Omega$, $\cc_{\Omega}$ coincides with the asymptotic cone of $\Omega$, giving \eqref{eq22} by \cite[Theorem~2.2.1]{AusTeb}. When $\Omega$ instead is open and convex we have that $\Omega$ is equal to its \textit{relative interior} $\ri(\Omega)$, and since $\cc_{\ri(\Omega)} = \cc_{\overline{\Omega}}$ \cite[Proposition 1.4.2]{Bert09}, it follows that $\cc_{\Omega} = \cc_{\overline{\Omega}}$ in this case.

We next recall some standard terminology and facts of polytopes, referring to for example \cite[Ch. 7--9]{brondsted2012introduction}. By an open halfspace in $\R^d$ we mean a set
$$H_\nu^r=\{x \in \R^d \, : \, x\cdot \nu>r\},$$
 where $\nu \in \R^d$ is a non-zero vector and $r\in\R$. A closed half-space is the closure of such a set. A finite intersection of half-spaces is called a polyhedral set.

 A convex polytope is a bounded polyhedral set. A closed convex polytope is the convex hull of a finite set of points. The minimal set of such points coincides with the extreme points of the polytope, that is, its vertices.  If the minimal number of defining hyperspaces of a convex polytope is $d+1$ (equivalently, if it has precisely $d+1$ vertices), the polytope is called a simplex. For a non-closed polytope we define its vertices (and its edges and facets) as those of its closure.

The boundary of a polytope set is made up of a finite amount of facets (i.e. $d-1$ dimensional faces), see Corollary 7.4 and Theorem 8.1 of \cite{brondsted2012introduction}. For a polytope $\Pi$ with vertex $x_j$, we denote by $\partial_{\far,x_j}\Pi$ the part of its boundary made up of all facets not containing $x_j$.

A vertex of a polytope will be called simple if it is contained in precisely $d$ of its edges. We say that a polytope is simple if all of its vertices are simple, which coincides with the standard terminology. Equivalently, this means that each vertex is contained in precisely $d$ of its facets (cf. \cite[Theorem~12.11]{brondsted2012introduction}).

By an affine linear transformation we mean a map of the form $A(x)=x_0+L(x)$ where $L$ is a linear map, and we call $x_0$ the origin of such a map. The following simple lemma gives a third characterization of simple vertices.

\begin{lemma}\label{l3}
Let $\{x_j\}_{j=1}^J$ be the vertices of a closed polytope $\Pi$. Then the vertex $x_j$ is simple if and only if it is the origin of an invertible affine transformation $A_j$ such that $\Pi$ locally coincides with $A_j(\overline{\R_+^d})$ around $x_j$, i.e. for any neighborhood $U$ of $x_j$ such that $\overline{U} \cap \partial_{\far, x_j}\Pi = \emptyset$ we have that
$$A_j^{-1}(\Pi \cap U) = \overline{\R_+^d} \cap A_j^{-1}(U).$$
\end{lemma}
By compactness it is easy to construct a partition of unity adapted to the vertices of $\Pi$.
\begin{lemma}\label{l9}
Given a polytope $\Pi$ with vertices $\{x_j\}_{j=1}^J$ there exist functions $\{\mu_j\}_{j=1}^J$ such that $\mu_j \in C_c^\infty(\R^d)$, $\sum_{j=1}^J \mu_j(x) = 1$ for $x \in \overline{\Pi}$, and $\supp\mu_j\cap\partial_{\far, x_j}\Pi=\emptyset$.
\end{lemma}

\section{General domain Hankel operators} \label{sec:gdhankel}

We now consider general domain Hankel operators $\Gamma_{f,\Xi}$ for convex domains $\Xi$. Observe that in this case $\Omega=\Xi+\Xi = 2\Xi$. We begin with a proposition that links the bounded Hankel operators with weak factorization.

\begin{proposition} \label{lem:banach}
	Let $\Xi$ be an open convex domain. Then
	$$X = \left\{\Gamma_{f, \Xi} \, : \, \| \Gamma_{f, \Xi} \| < \infty \right\}$$
	is a closed subspace of the space of bounded linear operators on $L^2(\Xi)$. As a Banach space, it is isometrically isomorphic to the dual space $(\PW_\Xi \odot \PW_\Xi)^*$. More precisely, bounded functionals $\mu$ on the projective tensor product correspond to distributions $f$ on $\Omega = 2\Xi$,
$$( f, g ) = \mu(\mathcal{F}^{-1}g), \quad g \in C_c^\infty(\Omega),$$
for which $\|\Gamma_{f, \Xi}\| = \|\mu\|.$
\end{proposition}
\begin{proof}
The main fact to be proved is that
$$\mathcal{F}^{-1}(C_c^\infty(\Omega)) \subset  \PW_{\Xi}\odot \PW_{\Xi}.$$
Since $C_c^\infty(\Xi)$ is dense in $L^2(\Xi)$, it then follows that $\mathcal{F}^{-1}(C_c^\infty(\Omega))$ is dense in the product $\PW_{\Xi}\odot \PW_{\Xi}$.

We will actually show a little more than the claim. Namely, every $g \in C_c^\infty(\Omega)$ can be written
$$g = \sum_{k} g_k * h_k, \quad g_k, h_k \in L^2(\Xi),$$
in such a way that the corresponding map $g \mapsto \sum_k \|g_k\|_{L^2(\Xi)} \|h_k\|_{L^2(\Xi)}$ is continuous from $C_c^\infty(\Omega)$, equipped with the usual test function topology, to $\mathbb{R}$. By employing a partition of unity in which each member is compactly supported in a cube, it is sufficient to prove the claim when $\Xi = (0, 1/2)^d$. For this we employ Fourier series. Let $\lambda(t) = 1/2 - |t-1/2|$, $t \in [0,1]$, and let
$$\Lambda (x) = \prod_{i=1}^d \lambda(x_i), \quad x \in (0,1)^d.$$
Note that $\lambda$ is in the Wiener algebra $A([0,1])$, the space of functions on $[0,1]$ with absolutely convergent Fourier series, equipped with pointwise multiplication. Therefore $\Lambda$ is in the Wiener algebra $A([0,1]^d)$, since $\Lambda$ is a tensor power of $\lambda$. Since $g \in C_c^\infty((0,1)^d)$ and $\Lambda$ is non-zero on compact subsets of $(0,1)^d$ it follows by Wiener's lemma \cite[Ch. 5]{Groch10} that $g/\Lambda \in A([0,1]^d)$ (to apply Wiener's lemma, first modify $\Lambda$ to be nonzero outside the support of $g$). Expanding $g/\Lambda$ in a Fourier series,
$$(g/\Lambda)(x) = \sum_{k \in \Z^d} a_k e^{i2\pi k \cdot x}, \quad \sum_{k \in \Z^d} |a_k| < \infty, \quad x \in [0,1]^d,$$
let $h_k(x) = e^{i2\pi k \cdot x} \chi_{(0,1/2)^d}(x)$, $g_k = a_k h_k$. Then a computation shows that
$$(g_k * h_k)(x) = a_k e^{i2\pi k \cdot x} \Lambda(x), \quad x \in (0,1)^d,$$
so that
$$g = \sum_{k \in \Z^d} g_k * h_k, \quad \sum_{k \in \Z^d} \|g_k\|_{L^2((0,1/2)^d)} \|h_k\|_{L^2((0,1/2)^d)} < \infty.$$
An inspection of the argument shows that the $g \mapsto g/\Lambda$ is continuous from $C_c^\infty((0,1)^d)$ to $A([0,1]^d)$, and therefore $g \mapsto \sum_k \|g_k\|_{L^2((0,1/2)^d)} \|h_k\|_{L^2((0,1/2)^d)}$ is continuous on $C_c^\infty((0,1)^d)$ as promised.

Suppose now that $\mu \in (\PW_{\Xi}\odot \PW_{\Xi})^*$. We have just demonstrated that $( f, g ) =  \mu(\mathcal{F}^{-1}g)$, $g \in C_c^\infty(\Omega)$, defines a distribution on $\Omega$. Hence we may consider the Hankel operator $\Gamma_{f, \Xi}$. For $g, h \in C_c^\infty(\Xi)$ we have that
\begin{equation} \label{eq:muf}
\langle \Gamma_{f}g, h \rangle_{L^2(\Xi)} = ( f, g * \bar{h} ) = \mu(\mathcal{F}^{-1}g \cdot \mathcal{F}^{-1}\bar{h}).
\end{equation}
Since $\mu$ is a bounded functional on $\PW_{\Xi}\odot \PW_{\Xi}$ we conclude that
$$|\langle \Gamma_{f}g, h \rangle_{L^2(\Xi)}| \leq \|\mu\| \|\mathcal{F}^{-1} g\|_{\PW_\Xi} \|\mathcal{F}^{-1}\bar{h}\|_{\PW_\Xi} = \|\mu\| \|g\|_{L^2(\Xi)} \|h\|_{L^2(\Xi)},$$
that is, $\Gamma_{f, \Xi}$ is bounded, and in fact $\|\Gamma_f\| = \|\mu\|$. Conversely, if $f$ is a distribution on $\Omega$ such that $\Gamma_{f, \Xi}$ is bounded, it is clear that $f$ induces a bounded functional $\mu$ on $\PW_{\Xi}\odot \PW_{\Xi}$ by \eqref{eq:muf}. This proves that $X$ is isometrically isomorphic to the Banach space $(\PW_{\Xi}\odot \PW_{\Xi})^*$, which also entails that $X$ is closed, completing the proof.
\end{proof}

In the remainder of this section we assume that $\Xi$ is a convex polytope. Next we prove Theorem \ref{t1alt} under the additional assumption that $f$ is supported around one simple vertex of $\Omega$.

\begin{proposition}\label{p1}
	Let $\Xi \subset \R^d$ be an open convex polytope, $x$ a simple vertex of $\Omega=2\Xi$, and let $f\in \ddd'(\Omega)$ be such that $\supp f \cap \partial_{\far,x} \Omega = \emptyset$. If $\Gamma_f$ is bounded as an operator on $L^2(\Xi)$, then there exists a $b\in L^\infty(\R^d)$ such that $\hat b |_{\Omega}=f.$
\end{proposition}
\begin{proof}
	As in Lemma~\ref{l3}, let $A$ be an affine transformation with origin $x$ such that $A(\R_+^d)$ coincides with $\Xi$ in a neighborhood of $x$. It is straightforward to verify that it suffices to prove the proposition for $\Gamma_{f\circ A, A^{-1}(\Xi)}$. Since $A^{-1}(\Xi)$ is also a convex polytope, we may hence assume that $x=0$ and that $\Omega$ coincides with $\R_+^d$ in a neighborhood $U$ of $\supp f$, $\overline{U} \cap \partial_{far,0} \Omega = \emptyset$.  In particular, since $\Omega$ is a convex polytope, we have that $\Omega \subset \mathbb{R}_+^d$. Since $\supp f \subset U \cap \overline{\Omega}$ and $\overline{U} \cap \partial_{far,0} \Omega = \emptyset$, we can extend $f$ to a distribution on all of $\R_+^d$ by letting it be zero outside $\Omega$. Our strategy is to show that the operator $\Gamma_{f,\R_+^d}$ is bounded and to then apply Theorem~\ref{lacey}.
	
	For $n\in\N^d$ let $C_n$ denote the cube $(n_1,n_1+1)\times \ldots\times (n_d,n_d+1)$. For a set $X \subset \R_+^d$, let $P_X \colon L^2(\R_+^d) \to L^2(\R_+^d)$ denote the orthogonal projection of $L^2(\R_+^d)$ onto $L^2(X)$, and let $r>0$ be such that $$2\sqrt{d}r<\dist(U \cap \overline{\Omega},\partial_{far,0}\Omega).$$
	By considering test functions $g \in C_c^\infty(\R_+^d)$ such that $\supp g \cap \overline{rC_m} \subset rC_m$ for every $m$, we give meaning to the equality
	$$\Gamma_{f,\R_+^d}=\left(\sum_{n\in\N^d}P_{rC_n}\right)\Gamma_{f,\R_+^d} \left(\sum_{m\in\N^d}P_{rC_m}\right)=\sum_{m,n\in\N^d}P_{rC_n}\Gamma_{f,\R_+^d} P_{rC_m},$$
	a term $P_{rC_n}\Gamma_{f,\R_+^d} P_{rC_m}$ being non-zero only if
	\begin{equation}\label{empty}(rC_m+rC_n)\cap \supp f \neq \emptyset.\end{equation}
	Hence there are only finitely many non-zero terms in the decomposition. Since
	$$\|P_{rC_n}\Gamma_{f,\R_+^d} P_{rC_m}\|=\|\Psi_{f,rC_m,rC_n}\|,$$
	recalling the definition of $\Psi_f$ from Section~\ref{sec:background},
	it therefore suffices to prove that $\|\Psi_{f,rC_m,rC_n}\|$ is bounded whenever \eqref{empty} holds. If $rC_m,rC_n\subset \Xi$ there is nothing to prove since $\Gamma_{f,\Xi}$ is bounded by hypothesis. For the other terms, note that \eqref{empty}, $\supp f \subset U \cap \overline{\Omega}$, and the choice of $r$ implies that
	\begin{equation}\label{inc}rC_m+rC_n\subset \Omega,\end{equation}
	since $2\sqrt{d} r$ is the diameter of $rC_m+rC_n$. For any $z\in\R^d$, $x \in rC_n$, and $g \in C_c^\infty(rC_m)$ we have that $$\Psi_{f,rC_m,rC_n}(g)(x)=\int_{rC_m} f(x+y)g(y) \, dy=\int_{rC_m+z} f(x+(y-z))g(y-z)dy,$$
	and hence
	$$\|\Psi_{f,rC_m,rC_n}\|=\|\Psi_{f,rC_m+z,rC_n-z}\|.$$
	In particular, for $z=r(n-m)/2$ we obtain that $$\|\Psi_{f,rC_m,rC_n}\|=\|\Psi_{f,rC_{\frac{m+n}{2}},rC_{\frac{m+n}{2}}}\|.$$ However, $2rC_{\frac{m+n}{2}}=rC_m+rC_n$ so by \eqref{inc} we conclude that $rC_{\frac{m+n}{2}}\subset \Xi$. The desired boundedness now follows as it did in the first case considered.
	
	We have just demonstrated that $\|\Gamma_{f, \R_+^d}\| < \infty$. By Theorem~\ref{lacey} there exists a function $b \in L^\infty(\R^d)$ such that $\hat b|_{\R^d_+} = f$. This in particular implies that $\hat b|_{\Omega} = f$ when we return to the initial interpretation of $f$ as a distribution on $\Omega$.
\end{proof}

\begin{theorem} \label{t1alt}
	Let $\Xi$ be a simple convex polytope, and let $f\in \ddd'(\Omega)$, $\Omega=2\Xi$. Then $\Gamma_f \colon L^2(\Xi) \to L^2(\Xi)$ is bounded if and only if there is a function $b \in L^\infty(\R^d)$ such that $\hat b |_{\Omega}=f.$ There exists a constant $c > 0$, depending on $\Xi$, such that $b$ can be chosen to satisfy
	$$c\|b\|_{L^\infty} \leq \|\Gamma_f\| \leq \|b\|_{L^\infty}.$$
\end{theorem}

\begin{proof}
	Assume that $\Gamma_f$ is bounded. Let $\{x_j\}_{j=1}^J$ be the vertices of $\Omega$, and let $\{\mu_j\}_{j=1}^J$ be partition of unity as in Lemma \ref{l9}. For $\varphi \in C_c^\infty(\Xi)$ and $x \in \Xi$ we have that
	$$\Gamma_{\mu_j f}(\varphi)(x)=\int_{\Xi} \int_{\R^d} \hat{\mu}_j(\xi)e^{2\pi i (x+y)\cdot\xi}\,d\xi \, f(x+y)\varphi(y) \, dy= \int_{\mathbb{R}^d} \hat{\mu}_j(\xi){e^{2\pi i \xi \cdot x}}\Gamma_{f}(\varphi_\xi)(x) \, d\xi,$$
	where $\varphi_\xi(y) = e^{2\pi i y\cdot\xi}\varphi(y)$. Hence, $\Gamma_{\mu_j f} \colon L^2(\Xi) \to L^2(\Xi)$ is bounded,
	$$\|\Gamma_{\mu_j f}\| \leq \|\hat{\mu}_j\|_{L^1} \|\Gamma_f\|.$$
	Therefore, by Proposition~\ref{p1} there are functions $b_j \in L^\infty$ such that $\mu_jf=\hat b_j|_\Omega$. Thus $f=\hat{b}|_\Omega$, where $b = \sum_{j=1}^J b_j \in L^\infty$. Conversely, if $f=\hat{b}|_\Omega$, where $b \in L^\infty$, then $\Gamma_f$ is bounded by Proposition \ref{ptrivial}.

	The constant $c$ now arises from abstract reasoning. Consider the Banach space
$$X = \left\{\Gamma_{f, \Xi} \, : \, \| \Gamma_{f, \Xi} \| < \infty \right\}$$
of Proposition~\ref{lem:banach}. We have just shown that $b \mapsto \Gamma_{\hat b |_{\Omega}, \Xi}$ is a map of $L^\infty$ onto $X$. The open mapping theorem hence guarantees the existence of $c$.
\end{proof}
We immediately obtain the corresponding result for Toeplitz operators, when $\Xi$ is a simple convex polytope which, possibly after a translation, is symmetric under $x \mapsto -x$.
\begin{corollary}\label{t2}
	Let $\Xi$ be a simple convex polytope such that for some $z \in \R^d$ it holds that $\Xi+z = -\Xi - z$. Let $f \in \ddd'(\Omega)$, $\Omega=\Xi - \Xi = 2\Xi + 2z$. Then $\Theta_f$ is bounded if and only if there exists a function $b \in L^\infty(\R^d)$ such that $\hat b|_{\Omega} = f$. There exists a constant $c > 0$, depending on $\Xi$, such that $b$ can be chosen to satisfy
	$$c\|b\|_{L^\infty} \leq \|\Theta_f\| \leq \|b\|_{L^\infty}.$$
\end{corollary}
\begin{proof}
In this case $\Theta_{f} g = \Gamma_{\tilde{f}} \tilde{g}$, where $\tilde{f}(x) = f(x + 2z)$, $x \in 2\Xi$, and $\tilde{g}(x) = g(-x-2z)$, $x \in \Xi$. Hence the result follows from Theorem~\ref{t1alt}.
\end{proof}
We also deduce the weak factorization result for $\PW_\Omega^1$, see Section~\ref{subsec:PW}.
\begin{corollary} \label{cor:fact}
	Let $\Xi$ be a simple convex polytope, and let $\Omega = 2\Xi$. Then
	$$PW_{\Omega}^1 = \PW_{\Xi}\odot \PW_{\Xi}.$$
	The norms of these Banach spaces are equivalent.
\end{corollary}
\begin{proof}
By Cauchy-Schwarz, the inclusion $I \colon \PW_{\Xi}\odot \PW_{\Xi} \to PW_{\Omega}^1$ is bounded. Since $I$ has dense range by Proposition~\ref{lem:banach}, the adjoint $I^* \colon (\PW^1_\Omega)^* \to (\PW_{\Xi}\odot \PW_{\Xi})^*$ has empty kernel. Suppose $\mu \in (\PW_{\Xi}\odot \PW_{\Xi})^*$.  Note that $CG(x) = \overline{G(-x)}$ defines an anti-linear isometric involution  $C \colon \PW_{\Xi}\odot \PW_{\Xi} \to \PW_{\Xi}\odot \PW_{\Xi}$. This induces an anti-linear isometric involution $D \colon (\PW_{\Xi}\odot \PW_{\Xi})^* \to (\PW_{\Xi}\odot \PW_{\Xi})^*$,
$$D\mu(G) = \overline{\mu(CG)}, \quad G \in \PW_{\Xi}\odot \PW_{\Xi}.$$

According to Proposition~\ref{lem:banach}, $( f, g ) =  \mu(\mathcal{F}^{-1}g)$, $g \in C_c^\infty(\Omega)$, defines a distribution on $\Omega$ such that $\|\Gamma_{f,\Xi}\| = \|\mu\|$. By Theorem~\ref{t1alt}, there is a function $b \in L^\infty(\R^d)$ such that $\hat{b}|_{\Omega} = f$. Since $\PW_\Omega^1 \subset L^1(\R^d)$, we can interpret $b$ as an element of $(\PW_\Omega^1)^*$, $b(G) = \langle G, b \rangle_{L^2(\R^d)}$. Then, recalling that $JG(x) = G(-x)$, we have that
$$DI^* b (G) = (b, JG)  = ( f, \mathcal{F}^{-1} JG ) = ( f, \mathcal{F} G ) = \mu(G), \quad G \in \mathcal{F}^{-1}(C_c^\infty(\Omega)),$$
that is, $DI^*b = \mu$, or $I^*b = D\mu$. Since $D$ is an involution, it follows that $I^*$ is onto.

In other words, $I^* \colon (\PW^1_\Omega)^* \to (\PW_{\Xi}\odot \PW_{\Xi})^*$ is a Banach space isomorphism, and therefore the inclusion $I \colon \PW_{\Xi}\odot \PW_{\Xi} \to \PW^1_\Omega$ is as well. Hence,
$$\PW_{\Xi}\odot \PW_{\Xi} = \PW^1_\Omega,$$
and the norms of these two Banach spaces are equivalent, by the open mapping theorem.
\end{proof}
The method used to prove Theorem~\ref{t1alt} extends to many unbounded polyhedral sets. Instead of pursuing a general statement, let us consider the example of a strip in $\R^2$,
\begin{equation} \label{eq:strip}
\Xi = \R_+ \times (0,1).
\end{equation}
This is an interesting addition to Theorem~\ref{t1alt}, since $\Xi$ does not have a simple vertex at infinity. In fact, $\partial \Xi$ may be considered to have a cusp point there.

\begin{proposition}
Let $\Xi$ be the strip defined in \eqref{eq:strip}, and let $f \in \ddd'(\Omega)$, $\Omega = 2\Xi$. Then $\Gamma_f \colon L^2(\Xi) \to L^2(\Xi)$ is bounded if and only if there is a function $b \in L^\infty(\R^d)$ such that $\hat b |_{\Omega}=f.$
\end{proposition}
\begin{proof}[Proof sketch.]
Let $\nu_1, \nu_2 \in C_c^\infty(\R)$ be functions such that $\nu_1(t) + \nu_2(t) = 1$ for $t \in [0,2]$, $\nu_1$ vanishes in a neighborhood of $2$, and $\nu_2$ vanishes in a neighborhood of $0$. Let
$$\mu_j(x) = \nu_j(x_2), \quad j=1,2, \; x = (x_1, x_2) \in \Omega.$$
Then for $\varphi \in C_c^\infty(\Xi)$ and $x = (x_1, x_2) \in \Xi$ we have that
	$$\Gamma_{\mu_j f}(\varphi)(x)=\int_{\Xi} \int_{\R} \hat{\nu}_j(\xi)e^{2\pi i (x_2+y_2) \xi}\,d\xi \, f(x+y)\varphi(y) \, dy= \int_{\mathbb{R}} \hat{\nu}_j(\xi){e^{2\pi i x_2\xi  }}\Gamma_{f}(\varphi_\xi)(x) \, d\xi,$$
	where $\varphi_\xi(y) = e^{2\pi i y_2 \xi}\varphi(y)$, $y = (y_1, y_2) \in \Xi$. Hence, as before we see that
	\begin{equation} \label{eq:stripbdd}
	\|\Gamma_{\mu_j f, \, \Xi}\| \leq \|\hat{\nu}_j\|_{L^1} \|\Gamma_{f, \, \Xi} \|, \quad j=1,2.
	\end{equation}
As in Proposition~\ref{p1} and Theorem~\ref{t1alt} it is sufficient to see that $\Gamma_{\mu_1 f} \colon L^2(\R_+^2) \to L^2(\R_+^2)$ and $\Gamma_{\mu_2 f} \colon L^2(\R_+ \times (-\infty, 1) ) \to L^2(\R_+ \times (-\infty, 1))$ define bounded operators, and by symmetry it is sufficient to consider the first of the two.

For $n \in \N$, let $S_n$ denote the strip $\R_+ \times (n, n+1)$, and let $r > 0$ be such that
$$
2r < \dist([0,2] \cap \supp \nu_1, 2).
$$
We decompose $\Gamma_{\mu_1 f} \colon L^2(\R_+^2) \to L^2(\R_+^2)$ according to strips instead of cubes,
$$
\Gamma_{\mu_1 f, \, \R_+^2} = \sum_{m,n \in \N} P_{rS_n} \Gamma_{\mu_1 f, \, \R_+^2} P_{r S_m}.$$
There are only a finite number of non-zero terms in this decomposition, and for any such term we by our choice of $r$ that
\begin{equation} \label{eq:stripcontained}
rS_m + rS_n \subset \Omega.
\end{equation}
For $n,m$ corresponding to a non-zero term, we have that
$$\|P_{rS_n} \Gamma_{\mu_1 f, \, \R_+^2} P_{r S_m} \| = \|\Psi_{\mu_1f, rS_m, rS_n}\| = \|\Psi_{\mu_1f, rS_m + z, rS_n - z}\| = \|\Psi_{\mu_1f, rS_{\frac{m+n}{2}}, rS_{\frac{m+n}{2}}}\|,$$
where $z = (0, r(n-m)/2)$. Since $rS_{\frac{m+n}{2}} \subset \Xi$ by \eqref{eq:stripcontained} and $\Gamma_{\mu_1 f} \colon L^2(\Xi) \to L^2(\Xi)$ is bounded by \eqref{eq:stripbdd}, we conclude that each non-zero term $P_{rS_n} \Gamma_{\mu_1 f, \, \R_+^2} P_{r S_m}$ is bounded. Hence $\Gamma_{\mu_1 f} \colon L^2(\R_+^2) \to L^2(\R_+^2)$ is bounded, finishing the proof.
\end{proof}
\section{General domain Toeplitz operators} \label{sec:gdtoep}

In this section we consider general domain Toeplitz operators on open convex domains $\tilde\Xi\subset\R^d$ such that both $\cc_{\tilde\Xi}$ and $\bc_{\tilde\Xi}$ have non-empty interior (as in the classical case $\tilde\Xi=\R_+$). This forces $\tilde\Xi$ to be unbounded and, as we shall soon see, it also entails that $\tilde{\Omega}=\tilde\Xi-\tilde\Xi = \R^d$. We shall also consider more general open connected sets $\Xi$ such that there are points $x_0$ and $x_1$ for which \begin{equation}\label{gd}x_1+\tilde\Xi\subset \Xi\subset x_0+\tilde\Xi,\end{equation}
and prove that $\|\Theta_{f, \Xi}\|=\|\hat f\|_{L^\infty}$ under this hypothesis. This allows for domains $\Xi$ with very irregular boundaries, in sharp contrast to Theorem \ref{t1alt}. The corresponding class of operators $\Theta_{f,\Xi}$ partially extends the class of generalized Toeplitz operators considered in \cite{janson1988paracommutators}, see Section~\ref{secBHO}. The next theorem can also be recovered by verifying the hypotheses of and keeping track of the constants in the proof of \cite[Theorem 5.4]{janson1988paracommutators}. However, for completeness we prefer to give our own concrete proof.

\begin{theorem}\label{toeplitz}
Let $\Xi$ be a set as above. Then $\Xi-\Xi=\R^d$ and, for $f\in\ddd'(\R^d)$, we have that $\Theta_f \colon L^2(\Xi) \to L^2(\Xi)$ is bounded if and only if $f\in \f^{-1}(L^\infty)$. Moreover, $\|\Theta_f\|=\|\hat f\|_{L^\infty}$.
\end{theorem}
\begin{proof}
Fix $z\in\R^d$ and set $|z|=R$. Pick a vector $e\in \inte (\cc_{{\tilde{\Xi}}})$ with distance greater than $R$ to the complement of $\cc_{\tilde{\Xi}}$, which is possible since $\cc_{{\tilde{\Xi}}}$ is a cone with non-empty interior. Then $e+z\in \cc_{{\tilde{\Xi}}}$, so for any $x \in\tilde{\Xi}$ we have that $x_1+x+e+z\in x_1+\tilde{\Xi}\subset \Xi$. Similarly, $x_1+x+e\in \Xi$. Since $z$ is the difference of these two vectors, the first claim follows.

Suppose that we have proven the theorem for all $f \in C_c^\infty(\R^d)$. If $f$ is a general symbol for which $\Theta_f$ is bounded, consider the sequence of functions $f_n \in C_c^\infty(\R^d)$ from Proposition~\ref{pextend2}. Then $\hat f_n$ has, by Alaoglu's theorem, a subsequence $\hat f_{n_k}$ which converges weak-star in $L^\infty$ to some element $g$. Since $f_n$ converges to $f$ in distribution, it must be that $g = \hat{f}$. Hence $f\in \f^{-1}(L^\infty)$ and, by Propositions~\ref{ptrivial} and \ref{pextend2}, we have that
$$\| \hat f \|_{L^\infty} \leq \varlimsup_{k \to \infty} \| \hat f_{n_k} \|_{L^\infty} = \varlimsup_{k \to \infty} \|\Theta_{f_{n_k}}\| \leq \|\Theta_{f}\| \leq \| \hat f \|_{L^\infty}.$$
This proves the theorem for general symbols.

 Hence we assume that $f \in C_c^\infty(\R^d)$. Fix $\xi\in\R^d$, pick any vector $\nu$ in $\inte(\bc_{\tilde\Xi})$, and consider for $\varepsilon > 0$ the function $$E_{\varepsilon}(x)=e^{\varepsilon x\cdot \nu+2\pi i x\cdot\xi}\chi_{\Xi}(x), \quad x \in \Xi.$$ By \cite[Lemma~9.5]{andersson2015general} this function is in $L^2(x_0+\tilde\Xi)$,\footnote{The set $\bc_{\tilde\Xi}$ was denoted $\Theta$ in \cite{andersson2015general}.} and hence $E_\varepsilon \in L^2(\Xi)$. We use $E_\varepsilon$ as a test function:
\begin{align*}
\|\Theta_f\| \geq \left|\frac{\scal{\Theta_f E_{\varepsilon},E_{\varepsilon}}}{\|E_\varepsilon\|^2}\right| &= \left|\frac{1}{\|E_\varepsilon\|^2}\int\int f(x-y)e^{\varepsilon (x+y)\cdot \nu}e^{2\pi i (y-x)\cdot \xi}\chi_{\Xi}(y)\chi_{\Xi}(x) \, dy \, dx\right| \\
&=\left|\frac{1}{\|E_\varepsilon\|^2}\int f(z)e^{-2\pi i z\cdot \xi}\int e^{\varepsilon (z+2y) \cdot \nu}\chi_{\Xi}(z+y)\chi_{\Xi}(y) \, dy \,dz\right|
\end{align*}
Hence it follows that $\|\Theta_f\|\geq |\hat{f}(\xi)|$ upon showing that \begin{equation}\label{ef}\lim_{\varepsilon\rightarrow 0^+} \frac{e^{\varepsilon z \cdot \nu}}{\|E_\varepsilon\|^2} \int e^{2\varepsilon y \cdot \nu}\chi_{\Xi}(z+y)\chi_{\Xi}(y) \, dy = 1\end{equation}
uniformly on compacts in $z$. Since $\xi$ is arbitrary this establishes that $\|\Theta_f\|\geq\|\hat f\|_{L^\infty}$ and by Proposition \ref{ptrivial} we then conclude that $\|\Theta_f\|=\|\hat f\|_{L^\infty}$.

Fix $R>0$ and suppose that $z \in \R^d$ with $|z|<R$. Again, pick a vector $e\in \inte (\cc_{{\tilde{\Xi}}})$ with distance greater than $R$ to the complement of $\cc_{{\tilde{\Xi}}}$. Then $e+z\in \cc_{{\tilde{\Xi}}}$, and therefore
 $$-z+\Xi\supset -z+(x_1+\tilde\Xi)\supset -z+x_1+(e+z)+\tilde\Xi\supset x_1+e-x_0+x_0+\tilde\Xi \supset x_1+e-x_0+\Xi.$$
 With $x_2=x_1+e-x_0$ we have just shown that $x_2+\Xi\subset-z+\Xi$. It also holds that $x_2+\Xi\subset\Xi$,  by the last inclusion in the above chain and the fact that $x_1+e+\tilde\Xi\subset x_1+\tilde\Xi\subset\Xi$. This gives us that
$$ \chi_{\Xi}(y-x_2)=\chi_{\Xi}(y)\chi_{\Xi}(y-x_2)\leq\chi_{\Xi}(y)\chi_{\Xi}(y+z)\leq \chi_{\Xi}(y),$$ and hence that
$$e^{\varepsilon 2x_2 \cdot \nu}\|E_\varepsilon\|^2 =\int e^{2\varepsilon y \cdot \nu}\chi_{\Xi}(y-x_2) \, dy\leq\int e^{2\varepsilon y \cdot \nu}\chi_{\Xi}(y+z) \chi_{\Xi}(y) \, dy\leq \int e^{2\varepsilon y \cdot \nu}\chi_{\Xi}(y)\,dy =\|E_\varepsilon\|^2.$$ The desired equality \eqref{ef} is now immediate, completing the proof.
\end{proof}

\begin{corollary} \label{cor:ex}
	Let $\Xi \subset \R^d$ be any open connected domain such that
	$$(1,\infty)^d \subset \Xi \subset (0,\infty)^d,$$
	and let $f \in \ddd'(\R^d)$. Then $\Theta_{f} \colon L^2(\Xi) \to L^2(\Xi)$ is bounded if and only if $f$ is a tempered distribution and $\|\hat{f}\|_{L^\infty(\R^d)} < \infty$, and in this case
	$$\|\Theta_f\| = \| \hat{f} \|_{L^\infty}.$$
\end{corollary}

\section{Bounded extension of multi-level block Toeplitz/Hankel-matrices} \label{sec:blocktoep}

In this section we interpret Corollary~\ref{t2}, when $\Xi$ is a $d$-dimensional cube, as a result on the possibility of extending finite multi-level block Toeplitz matrices to infinite multi-level block Toeplitz matrices which are \textit{bounded} as operators on $\ell^2$. In view of the equivalence between Toeplitz and Hankel operators on the cube (cf. the proof of Corollary~\ref{t2}), and a similar equivalence for finite Hankel and Toeplitz matrices, we could equally well make the analogous statement for multi-level block Hankel matrices. We present only the Toeplitz-case. Such matrices appear in various applications, for example in multi-dimensional frequency estimation. Note in particular that Pisarenko's famous method for one-dimensional frequency estimation \cite{pisarenko1973retrieval}, which relies on the classical Carath\'eodory-Fej\'er theorem, was recently extended to the multi-variable case \cite{yang2015generalized} (see also \cite{andersson2016structure}).

When $d=1$ our statement reduces to a well-known theorem on extending finite (ordinary) Toeplitz matrices, appearing previously for example in \cite{BT01} and \cite{NF03}. To describe it, recall that a finite $N\times N$ Toeplitz-matrix is characterized by its constant diagonals, whose values we denote by $a = (a_{-N+1},\ldots a_{N-1})$. As an operator $T_a$ on $\ell^2(\{0,\ldots,N-1\})$, its action is given by
 $$T_a(v)(m)=\sum_{n=0}^{N-1}a_{m-n}v_n, \quad v \in \ell^2(\{0,\ldots,N-1\}), \; m \in \{0,\ldots,N-1\}.$$
We can also consider the case when $N = \infty$, the definitions extending in the obvious way. The completion result then states that it is always possible to extend $a$ to a \textit{bi-infinite} sequence $\tilde{a}$ such that the corresponding Toeplitz operator $T_{\tilde{a}}\colon \ell^2(\N) \to \ell^2(\N)$ satisfies $$\|T_{\tilde{a}}\|\leq 3\|T_a\|.$$
It is an open problem whether the constant $3$ is the best possible in this inequality. A discussion offering different approaches to the optimal constant can be found in \cite{Bess152}. See also \cite{Sara07}.

When $d>1,$ each multi-sequence $a=(a_n)_{n\in \{-N+1,\ldots,N-1\}^d},$ generates a multi-level block Toeplitz matrix $T_a$. As an operator on $\ell^2(\{0,\ldots,N-1\}^d)$ it is given by the formula
$$T_a(v)(m)=\sum_{n\in \{0,\ldots,N-1\}^d}a_{m-n}v_n, \quad v \in \ell^2(\{0,\ldots,N-1\}^d), \; m \in \{0,\ldots,N-1\}^d.$$
To understand this matrix, consider the $d$-level block Toeplitz matrix $T_a$ as an ordinary $N \times N$-Toeplitz matrix with entries which are $(d-1)$-level block Toeplitz matrices,
$$T_a = \{A_{i-j}\}_{i,j \in \{0,\ldots,N-1\}}, \quad A_i = \{a_{(i,m-n)}\}_{m,n \in \{0,\ldots,N-1\}^{d-1}}.$$
For instance, a multi-level block Toeplitz matrix for $d=2$ is an $N\times N$ Toeplitz matrix whose entries are $N \times N$ Toeplitz matrices. Again, we allow for the possibility that $N=\infty$. We now provide the multi-level block Toeplitz matrix analogue of the Toeplitz matrix completion theorem.

\begin{theorem} \label{thm:blocktoep}
	There exists a constant $C_d > 0$ such that any finite multi-sequence $a$ can be extended to an infinite multi-sequence $\tilde a$ on $\Z^d$ for which $T_{\tilde{a}} \colon \ell^2(\N^d) \to \ell^2(\N^d)$ is bounded with norm
	$$\|T_{\tilde{a}}\|\leq C_d\|T_a\|.$$
\end{theorem}
\begin{proof}
	Let
	$$f=\sum_{n\in \{-N+1,\ldots,N-1\}^d} a_n\delta_n,$$
	where $\delta_n$ is the Dirac delta function at $n$,
	$$\delta_n(\varphi)=\varphi(n), \quad \varphi \in C_c^\infty(\R^d).$$
	 Set $\Xi=(0,N)^d$ and consider $\Theta_f=\Theta_{f,\Xi}$. Given $g\in C^\infty_c(\Xi)$, a short calculation shows that
	$$\Theta_f(g)(x)=\sum_{n\in \Z^d\cap (x-\Xi)} a_n g(x-n), \quad x \in (0, N)^d.$$
	 With $x=m+r$, where $m \in \{0,\ldots,N-1\}^d$ and $r\in[0,1)^d$,  this can be rewritten
  	$$\Theta_f(g)(m+r)=\sum_{k\in \{0,\ldots,N-1\}^d} a_{m-k}g(r+k).$$
  	 In other words, with $g_r=\{g(r+n)\}_{n\in \{0,\ldots,N-1\}^d}$, we have that $$\Theta_f(g)(m+r)=T_a(g_r)(m).$$
	Hence
	 $$\sum_{m\in \{0,\ldots,N-1\}^d}|\Theta_f(g)(m+r)|^2=\|T_a(g_r)\|^2\leq \|T_a\|^2\|g_r\|^2=\|T_a\|^2\sum_{m\in \{0,\ldots,N-1\}^d}|g(m+r)|^2.$$
	Integrating both sides over $r \in (0,1)^d$ gives us that $\|\Theta_f(g)\|^2\leq\|T_a\|^2\|g\|^2$. In other words, $\Theta_f \colon L^2(\Xi) \to L^2(\Xi)$ is bounded and
	 $$\|\Theta_f\|\leq \|T_a\|.$$
	  Noting that the constant $c$ in Corollary~\ref{t2} is invariant under homotheties, we find that there exists a distribution $\tilde{f} = \hat b \in \ddd'(\R^d)$, coinciding with $f$ on $(-N,N)^d$, such that
	  $$\|\Theta_{\tilde f,\R^d}\|\leq C_d\|T_a\|,$$
	   where $C_d$ only depends on the dimension $d$. Of course, $\Theta_{\tilde f,\R^d} \colon L^2(\R^d) \to L^2(\R^d)$ is nothing but the operator of convolution with $\tilde f$.
	
	  Now pick any function $\varphi\in C^\infty_c((-1/2,1/2)^d)$ with $\int |\varphi|^2 dx=1$ and consider the isometry $I \colon \ell^2(\N^d)\to L^2(\R^d)$ given by $$I v(x)=\sum_{n\in\N^d}v_n \varphi(x-n), \quad v \in \ell^2(\N^d), \; x \in \R^d.$$
	  Then
	  $$I^* g(n) = \int_{\R^d} g(x) \overline{\varphi(x-n)} \, dx, \quad g \in L^2(\R^d), \; n \in \N^d.$$
	It follows that
	$$I^*\Theta_{\tilde f} I v (m) = \sum_{n \in \N^d} \tilde{a}_{m-n} v_n, \quad v \in \ell^2(\N^d), \; m \in \N^d,$$
	where
	$$\tilde{a}_n = \int_{\R^d} \int_{\R^d} f(x-y + n) \varphi(y) \overline{\varphi(x)} \, dy \, dx, \quad n \in \Z^d.$$
	 That is, $I^*\Theta_{\tilde f}I = T_{\tilde{a}}$. It is clear by construction that $\tilde{a}$ is an extension of $a$,
	 $$\tilde{a}_n = a_n \int_{\R^d} |\varphi(y)|^2 \, dy = a_n, \quad n \in \{-N+1,\ldots,N-1\}^d.$$
	 This finishes the proof.
\end{proof}

\bibliographystyle{amsplain}
\bibliography{MCPerfekt5}

\end{document}